\documentclass[a4paper,11pt]{elsarticle}
\usepackage{amssymb}
\usepackage{amsthm}
\usepackage{hyperref}
\hypersetup{colorlinks=true,linkcolor=blue,citecolor=magenta}
\usepackage{amsmath}
\usepackage{amscd,enumitem}
\usepackage{verbatim}
\usepackage{eurosym}
\usepackage{float}
\usepackage{graphicx}
\usepackage[section]{placeins}
\usepackage{color}
\usepackage{dcolumn}
\usepackage[mathscr]{eucal}
\usepackage{bbm}
\usepackage[textheight=8.5in, textwidth=6.7in]{geometry}
\usepackage{multirow}
\usepackage{caption}
\usepackage{tikz-cd}
\usepackage{tikz}
\usetikzlibrary{decorations.markings}
\usetikzlibrary{matrix, calc, arrows}
\usepackage{adjustbox}

\newtheorem*{thm*}{Theorem}
\newtheorem*{conj*}{Conjecture}

\newtheorem{thm}{Theorem}[section]

\newtheorem{obs}[thm]{Observation}

\newtheorem*{rmk}{Remark}

\theoremstyle{definition}
\newtheorem{dfn}{Definition}[section]
\newtheorem{lem}{Lemma}[section]

\newcommand\restr[2]{{
  \left.\kern-\nulldelimiterspace 
  #1 
  \vphantom{\big|} 
  \right|_{#2} 
 }}

\numberwithin{equation}{section}
\begin{document}

\begin{abstract}
    In this paper, we exhibit an irreducible Markov chain $M$ on the edge $k$-colorings of bipartite graphs based on certain properties of the solution space. We show that diameter of this Markov chain grows linearly with the number of edges in the graph. We also prove a polynomial upper bound on the inverse of acceptance ratio of the Metropolis-Hastings algorithm when the algorithm is applied on $M$ with the uniform distribution of all possible edge $k$-colorings of $G$. A special case of our results is the solution space of the possible completions of Latin rectangles.
\end{abstract}


\begin{frontmatter}
\title{A Markov chain on the solution space of edge-colorings of bipartite graphs}
\author[mit]{Letong Hong}
\author[renyi,sztaki]{Istv\'an Mikl\'os}

\address[mit]{Department of Mathematics, Massachusetts Institute of Technology \\{\tt email:} clhong@mit.edu}

\address[renyi]{Alfr\'ed R\'enyi Institute of Mathematics, E\"otv\"os Lor\'and Research Network \\ {\tt email:} miklos.istvan@renyi.hu}

\address[sztaki]{Institute for Computer Science and Control (SZTAKI), E\"otv\"os Lor\'and Research Network}

\end{frontmatter}

\section{Introduction}

\begin{sloppypar}
 The edge $k$-coloring problem asks if there is a coloring ${C: E \rightarrow \{c_1, c_2, \ldots, c_k\}}$ of a simple graph $G= (V,E)$ such that for any adjacent edges $e_1$ and $e_2$, $C(e_1) \ne C(e_2)$. 
 In this paper, we consider the solution space of the edge $k$-coloring problem of bipartite graphs. 
 Edge colorings of bipartite graphs have a variety of applications: the round-robin tournament \cite{Burkeetal2004, Skiena2008}, where the two vertex sets correspond to players from two parties, the pre-determined edges correspond to different kinds of games, and the edge colors correspond to index of the round; the open shop scheduling problem \cite{Williamsonetal1997}, where one vertex set corresponds to the list of objects to be manufactured and the other corresponds to the machines, the edges correspond to existing tasks, and the edge colors correspond to distinct categories of tasks; the path coloring problem in fiber-optic communication \cite{ErlebachJansen2001}, where edges correspond to pairs of nodes (belonging to two channels) with communication needs and edge colors correspond to frequency of light, under the restriction that no two paths that share a common piece of fiber can use the same frequency, to prevent interference.
 
 In general, the edge $k$-coloring problem is NP-complete \cite{Holier1981}. It remains NP-complete if the problem is restricted to edge $3$-coloring of $3$-regular simple graphs \cite{Holier1981}, and this hardness result has been extended to edge $r$-coloring of $r$-regular graphs for all $r\geq 3$ \cite{LevenGalil1983}.
 
 By Vizing's theorem \cite{Vizing1964}, every graph is edge $k$-colorable with $k \ge \Delta+1$ colors, where $\Delta$ is the maximum degree of the graph. Vizing also proved that planar graphs with $\Delta \ge 8$ are edge $\Delta$-colorable \cite{Vizing1965}. Sanders and Zhao \cite{SandersZhao2001} showed that planar graphs with $\Delta = 7$ are also edge $7$-colorable. For cubic planar graphs, it can be decided in polynomial time if they are edge $3$-colorable or edge $4$-colorable \cite{Tait1978,Robertsonetal1996}. The optimal edge coloring of planar graphs with $4 \le \Delta \le 6$ remains an open question.
 
 Based on the theorem of K\"onig \cite{Konig1916}, every bipartite graph is edge $\Delta$-colorable. Efficient algorithms exists for finding such an edge $\Delta$-coloring in bipartite graphs \cite{ColeHopcroft1982,KapoorRizzi2000}.
 
  Although efficient algorithms exist for finding one particular solution, such a solution might have a bias inherited from the construction algorithm. In that case, one one would like to generate a random solution uniformly.
  Finding a random completion of a Latin rectangle might also be needed in partially redesigning an experimental design \cite{Lakic}. Sampling completions of Latin rectangles is also necessary to give good stochastic estimations on the number of Latin squares of a given size \cite{Aksenetal2017}.
 
 Sampling and counting are related problems.  For many counting problems, approximate counting and approximate uniform sampling have the same computational complexity \cite{JerrumValiantVazirani1986}. Therefore it is natural to consider the problem of counting the solutions to the edge $k$-coloring problem.
 It is easy to develop a dynamic programming algorithm to compute how many edge $k$-colorings are there in a graph with $\Delta \le 2$. On the other hand, computing the number of edge $k$-colorings for $r$-regular planar graphs is \#P-complete for all $k\ge r \ge 3$ \cite{Caietal2016}. The complexity of counting edge colorings of bipartite graphs remains open, although it is natural to conjecture that this counting problem is hard for bipartite graphs, too. Indeed, there is no known easily computable formula even for a very specific case, the number of Latin squares, which can be considered as edge $n$-colorings of the complete bipartite graph $K_{n,n}$.

 Even when the exact counting of combinatorial objects is proven to be \#P-complete, there might be efficient algorithms to approximate the number of such objects \cite{Jerrum2003,Miklos2019}.
 Therefore, it is natural to consider approximate sampling of solutions to the edge $k$-coloring problem. Markov chain Monte Carlo methods \cite{Metropolisetal1953,Hastings1970} have been widely used for approximate sampling \cite{Jerrum2003,Miklos2019}.
 
 In the Markov chain Monte Carlo framework, a Markov chain is designed that converges to the uniform distribution of the combinatorial objects. The standard way is to use the Metropolis-Hastings algorithm \cite{Metropolisetal1953,Hastings1970}, which tailors a primary Markov chain to another one converging to a prescribed equilibrium distribution. In each step of the Metropolis-Hastings algorithm, a new state of the Markov chain is proposed from the current state based on the primary Markov chain. This proposed state is accepted with a probability computed from the desired equilibrium distribution and the transition probabilities of the primary Markov chain. Such a Markov chain can be efficiently used for approximate sampling if 
 the convergence time grows only polynomially with the size of the problem instance.
 
 Jacobson and Matthews introduced a Markov chain that converges to the uniform distribution of Latin squares \cite{JacobsonMatthew1996}. Although it is conjectured that the Jacobson-Matthew Markov chain is rapidly mixing, it is not proved yet. A Markov chain similar to the Jacobson-Matthew Markov chain was developed by Aksen \emph{et al.} \cite{Aksenetal2017} that converges to the uniform distribution on the half-regular factorizations of complete bipartite graphs. Both Markov chains have large perturbations. That is, in the Jacobson-Matthew Markov chain, three not necessarily consecutive rows of a Latin square are changed at each step to get a new Latin square. Furthermore, the number of entries that might be changed in these three rows are not bounded. Similarly, the Markov chain developed by Aksen \emph{et al.} changes the color of edges incident to three vertices in one vertex class and an unbounded number of vertices on the other vertex class. Examples are known when the large perturbations cause small acceptance ratios in the Metropolis-Hastings algorithm and thus torpid mixing of the Markov chain \cite{mms2010}. However, it is proved that the inverse of the acceptance ratio is upper bounded by a polynomial of the size of the problem instance both in case of the Jakobson-Matthew Markov chain and in case of the Markov chain developed by Aksen \emph{et al.} Furthermore, small diameter is proved for both Markov chains. Although these properties do not guarantee rapid mixing, they strengthen the conjecture of rapid mixing.
 
 In this paper, we give perturbations that are sufficient to transform any edge $k$-coloring of a bipartite graph $G= (V,E)$ to another edge $k$-coloring of the same graph. These transformations can be the transition kernels of a Markov chain, and in fact, we construct such a Markov chain, $M(G,k)$, on which the following properties are proved. 
 \begin{enumerate}[label=(\alph*)]
     \item The diameter of $M(G,k)$ grows only linearly with $|E|$. More specifically, $6|E|$ perturbations are sufficient to transform any edge coloring of $G$ to any another one.
     \item When the Metorpolis-Hastings algorithm is applied on $M(G,k)$ to converge to the uniform distribution of edge $k$-colorings of $G$, the inverse of the acceptance ratio is upper bounded by a polynomial of the number of vertices in $G$. More specifically, the inverse of the acceptance ratio is upper bounded by $96|V|^2(|V|-1)$.
 \end{enumerate}
 The main contribution of our paper is to show that the perturbations invented by Jacobson and Matthew for the regular, $1$-factorization of the complete bipartite graphs then improved by Aksen \emph{et al}. for the half regular factorizations of the complete bipartite graph can be extended to non-regular and non-complete bipartite graphs.
 
 A special case of our results is the completions of Latin rectangles or equivalently, the edge $k$-coloring of $k$-regular bipartite graphs. In that special case, we can prove better results on both the diameter and on the upper bound of the inverse of the acceptance rate. Particularly, the diameter is upper bounded by $3|E|$, and the inverse of the acceptance ratio is upper bounded by $16|V|(|V|-1)$.

The paper is organized as follows. In Section \ref{section2}, we introduce basic definitions and propositions related to edge colorings and Markov chain Monte Carlo, in particular the Metropolis-Hastings algorithm. In Section \ref{section3}, we then explicitly construct the chain $M$ on the solution space of edge colorings of bipartite graphs and prove our main theorem that $M$ has the properties described above. Finally we dedicate Section \ref{section4} to apply our result to completion of Latin rectangles.

 \end{sloppypar}
 

\section{Preliminaries}\label{section2}

In this section, we introduce the almost edge $k$-colorings and an imprtant lemma on them. We also give a short description to the Metropolis-Hastings algorithm.

\subsection{Almost edge $k$-colorings}

We define formally edge $k$-colorings and almost edge $k$-colorings:
 \begin{dfn}\label{def:edge-coloring}
An \emph{edge $k$-coloring} of a graph $G=(V,E)$ is a map $C: E(G)\to \{c_1,c_2,\ldots, c_k\}$ such that there is no adjacent monochromatic edges.

In a map $C: E(G) \to \{c_1,c_2,\ldots, c_k\}$, a vertex $v\in V$ has a \emph{$(+c-c')$-deficiency} if there are exactly two $c$ edges incident to $v$, all other incident edges have different colors, and there is no $c'$ edge incident to $v$.

An \emph{almost edge $k$-coloring} of $G = (V,E)$ is a map $C: E(G)\to \{c_1,c_2,\ldots, c_k\}$ such that at most two vertices have deficiency and  all other vertices have no adjacent monochromatic edges.

If $c \in \{c_1, c_2, \ldots, c_k\}$ and $C$ is a coloring of $G$ then $\restr{G}{C,c}$ is the subgraph that contains only the edges colored by $c$.
\end{dfn}

Observe that a deficient vertex in an almost edge $k$-coloring might have both $(+c-c')$-deficiency and $(+c-c'')$-deficiency for some colors $c'\ne c''$ if the degree of the vertex is smaller than $k$. 

We show how to transform almost edge $k$-colorings to edge $k$-colorings.
\begin{lem}\label{lem:transf-from-almost-to-proper}
Let $C$ be an almost edge $k$-coloring of a bipartite graph $G$ with one or two deficient vertices. Then $C$ can be transformed to an edge $k$-coloring by flipping colors along at most two alternating walks with some colors $c$ and $c'$.
\end{lem}
\begin{proof}
Let $v_1$ be a $(+c-c')$-deficient vertex, and let $e$ and $f$ be its incident edges with color $c$. Consider the longest alternating walk with colors $c$ and $c'$ that contains $e$. We claim that this walk is a path or a walk that ends in the second visit of the other deficient vertex. Particularly, it cannot return to $v_1$. Indeed, returning to any vertex other than $v_1$ would mean a deficient vertex. However, above $v_1$, there is at most one more deficient vertex, $v_2$. If the walk returns to $v_2$, then let $e_1$ be the edge with which the walk arrives at the first visit of $v_2$, and let $e_2$ be the edge with which the walk arrives at the second visit of $v_2$. We claim that $e_1$ and $e_2$ has the same color since the graph is bipartite.

Also, this walk cannot return to $v_1$. It could be only with a $c$-edge, however, this would mean an odd cycle, which is a contradiction in a bipartite graph.

Therefore,  if this longest alternating walk is a path, that starts with edge $e$, and ends in a vertex either with a $c$-edge, and then this end vertex does not have an incident $c'$ edge or ends in a vertex with a $c'$ edge, and then this end vertex does not have a $c$ edge. By flipping the colors along this path, the $(+c-c')$-deficiency of $v_1$ is eliminated, and no new deficiency is introduced.

If this longest alternating walk ends at the first visit of the other deficient vertex, $v_2$, then $v_2$ has either $(+c-c')$ or $(+c'-c)$ deficiency. By flipping the colors along this path, both the $(+c-c')$-deficiency of $v_1$ an the deficiency of $v_2$ are eliminated, and no new deficiency is introduced.

If this longest alternating walk ends at the second visit of $v_2$, then flipping the colors along this walk eliminates the $(+c-c')$-deficiency of $v_1$, changes the type of deficiency of $v_2$ from $(+c-c'')$ to $(+c'-c'')$ or from $(+c'-c'')$ to $(+c-c'')$, and does not create any new deficiency.

If there is one more deficient vertex remaining, do the same procedure with it.
\end{proof}

\subsection{Markov chain Monte Carlo}
 In the Markov chain Monte Carlo framework, a Markov chain is designed that converges to the uniform distribution of the combinatorial objects. The standard way is to use the Metropolis-Hastings algorithm \cite{Metropolisetal1953,Hastings1970}, which tailors a primary Markov chain to another one converging to a prescribed equilibrium distribution. First we define the Markov chains, then the Metropolis-Hastings algorithm.
 
 \begin{dfn}
A \emph{discrete time, finite Markov chain} $X=(X_0,X_1,\ldots)$ is a sequence of random variables taking values on a finite state space $S$, such that $X_n$ depends on only $X_{n-1}$. That is, for all $m\ge 0$ and all possible sequence $i_0,i_1,\ldots,i_{m+1}\in S^{m+2}$,  $$P(X_{m+1}=i_{m+1}\ |\ X_0=i_0,\ldots X_m=i_m)=P(X_{m+1}=i_{m+1}\ |\ X_m=i_m).$$

For sake of simplicity, we will write $P(i_{m+1}|i_m)$ instead of $P(X_{m+1}=i_{m+1}\ |\ X_m=i_m)$.

An \emph{irreducible} Markov chain is a Markov chain with only one communicating class, i.e. for any two states $i,j\in S$, the probability of getting from $i$ to $j$ and from $j$ to $i$ (in multiple steps) are both positive.

A Markov chain is \emph{aperiodic} if all its states $i\in S$ satisfy that $$d_i:=\gcd \ \{t\ge 1|\ p_{ii}^{(t)}=P(\text{getting from  $i$ to itself via $t$ steps})>0\}=1,$$
where $\gcd$ stands for greatest common divisor.


\end{dfn}


A standard technique to make a Markov chain aperiodic is to have $P(i|i) \ne 0$ for all states $i$. Frequently, $P(i|i)$ is set at least $\frac{1}{2}$. Such a Markov chain is called \emph{Lazy Markov chain} \cite{Jerrum2003, Miklos2019}. The spectral analysis of Lazy Markov chains is technically simpler.

\begin{dfn}
For a Markov chain that converges to an equilibrium distribution $\pi$, it satisfies the \emph{detailed balance equation}, $\pi(i)P(j|i)=\pi(j)P(i|j)$.

\end{dfn}

The \emph{Markov chain Monte Carlo} (MCMC) method is a class of algorithms that achieves the goal of sampling from a probability distribution. One typical example is the Metropolis-Hastings algorithm \cite{Hastings1970, Metropolisetal1953}, described below.

\begin{dfn}
The \emph{Metropolis-Hastings} algorithm takes a Markov chain on a state space $X$ and transforms it into another one.

For a given function $g: X \rightarrow \mathbb{R}^{>0}$, and an irreducible, aperiodic Markov chain with transition probabilities such that for all pair of states $x,y\in X$ it holds that $P(y|x) \ne 0 \Rightarrow P(x|y)\ne 0$, a new Markov chain is constructed on the same state space. 
The state $x_i$ is defined based on $x_{i-1}$ in the following steps.
\begin{enumerate}
    \item Propose a candidate $y$ for the next sample $y\sim P(\cdot|x_{i-1})$, where $\sim$ stands for ``following distribution".
    
    \item Generate a random number $u\sim U(0,1)$, where $U(0,1)$ is the uniform distribution on the $[0,1]$ interval.
    
    \item The proposed state $y$ is \emph{accepted}, that is, $x_i=y$ if $u\le \frac{g(y)P(x_{i-1}|y)}{g(x_{i-1})P(y|x_{i-1})}$. Otherwise the proposed state is \emph{rejected}, that is, $x_i=x_{i-1}$.
\end{enumerate}
\end{dfn}

Observe that the probability of accepting a proposed state $y$ is
$$\min \left\{ 1,\frac{g(y)P(x|y)}{g(x)P(y|x)} \right\}.$$ The fraction $\frac{g(y)P(x|y)}{g(x)P(y|x)}$ is called the \emph{Metropolis-Hastings ratio}, and the probability is called \emph{acceptance probability}.

It is easy to prove that the Markov chain defined by the Metropolis-Hastings algorithm is irreducible, aperiodic, and reversible with respect to $\pi$, the probability distribution obtained after $g$ is normalized, and thus, it converges to distribution $\pi$ starting in an arbitrary state \cite{Jerrum2003,Miklos2019}. A variant of the Metropolis-Hastings algorithm is when for each pair of states $x$ and $y$, there are several ways to transform $x$ to $y$: $W = \{w_1, w_2, \ldots, w_t\}$, and there is a one-to-one mapping from $W$ to the set $W' = \{w'_1, w'_2, \ldots,w'_t\}$, the possible ways to transform $y$ back to $x$. Then the Metropolis-Hastings ratio can be changed to
$$
\frac{g(y)P(x,w'_i|y)}{g(x)P(y,w_i|x)}
$$
when $y$ is proposed from $x$ via the way $w_i$. Here $P(y,w_i|x)$ is the probability that $y$ is proposed from $x$ via way $w_i$. It can be proved that the so-obtained Markov chain still converges to $\pi$,  the probability distribution obtained after $g$ is normalized \cite{lunteretal2005}. In this paper, we will use this variant of the Metropolis-Hastings algorithm.

Such a Markov chain can be efficiently used for approximate sampling if 
the convergence time grows only polynomially with the size of the problem instance.

\section{A Markov chain Monte Carlo on edge $k$-colorings of a bipartite graph}\label{section3}

Below we define a Markov chain on the edge $k$-colorings of a bipartite graph $G = (V,E)$. We are going to prove that this Markov chain is irreducible, its diameter is at most $6|E|$, and when applied in a Metropolis-Hastings algorithm, the inverse of the acceptance ratio is upper bounded by a cubic function of $|V|$. 


\begin{dfn}\label{def:markov-chain}
Let $G$ be a bipartite graph, and let $k\ge \Delta$, where $\Delta$ is the maximum degree of $G$. We define a Markov chain $M(G,k)$ on the edge $k$-colorings of $G$. Let the current coloring be $C$. Then draw a random edge $k$-coloring in the following way.
\begin{enumerate}
    \item With probability $\frac{1}{2}$, do nothing (so the defined Markov chain is a Lazy Markov chain).
    \item With probability $\frac{1}{4}$, select two colors $c$ and $c'$ from the set of colors uniformly among the ${k\choose 2}$ possible unordered pairs. Consider the subgraph of $G$ with colors $c$ and $c'$; let it be denoted by $H$. Its non-trivial components are paths and cycles. Then select one of the following:
    \begin{enumerate}
        \item With probability $\frac{1}{3}$, do nothing.
        \item With probability $\frac{1}{3}$, uniformly select a sub-path or a cycle of $H$ from all subpaths and cycles in $H$.
        \item With probability $\frac{1}{3}$, uniformly select two sub-paths of $H$ from all possible pair of sub-paths such that both of them have exactly one end vertex which is also an end vertex in $H$. 
    \end{enumerate}
    Flip the colors in the selected subpaths. If there is any deficient vertex, select one of them uniformly, then select one of the edges with the same color incident to the selected vertex uniformly and change its color from $c$ to $c'$ or $c'$ to $c$. We denote this subcase by \emph{i)} for reference. Otherwise, with $\frac{1}{2}$ probability, select either a $c$-edge or a $c'$-edge uniformly from all edges with color $c$ and $c'$, and change its color to the opposite ($c'$ or $c$) (we denote this subcase by \emph{ii)} for reference) and with $\frac{1}{2}$ probability, do nothing. Eliminate all deficiency by flipping the colors on appropriate alternating path, selecting uniformly from the two neighbor edges with the same color incident to deficient vertices.
    \item With probability $\frac{1}{4}$ select three colors uniformly from the all possible ${k\choose 3}$ unordered pairs, and select uniformly one from the three colors. Let the distinguished color be denoted by $c''$, and let the other two colors be denoted by $c$ and $c'$. Consider the subgraph consisting of the colors $c$ and $c'$; let it be denoted by $H$. Its non-trivial components are paths and cycles. Then select one of the following:
    \begin{enumerate}
        \item With probability $\frac{1}{3}$, do nothing.
        \item With probability $\frac{1}{3}$, uniformly select a sub-path of $H$ from all subpaths such that at least one of its end vertices is not an end vertex in $H$.
        \item With probability $\frac{1}{3}$, uniformly select two sub-paths of $H$ from all possible pair of sub-paths such that both of them has exactly one end vertex which is also an end vertex in $H$.
    \end{enumerate}
    Flip the colors in the selected subpaths. If there are two deficient vertices, select one of them uniformly. If there is only one deficient vertex, then with probability $\frac{1}{2}$ select it, and with probability $\frac{1}{2}$, select uniformly a non-deficient vertex among those which are incident to both a $c''$ edge and at least one of a $c$ or a $c'$ edge. If there is no deficient vertex, then select uniformly a non-deficient vertex among those which are incident to both a $c''$ edge and at least one of a $c$ or a $c'$ edge.
    
    If a deficient vertex is selected, select uniformly one of the edges with the same color incident to the selected vertex, and let this color be denoted by $\tilde{c}$. Consider the maximal alternating path or cycle with colors $\tilde{c}$ and $c''$ that contains the selected edge. Flip the colors in this path or cycle. We denote this subcase by \emph{iii)} for reference.
    
    If a non-deficient vertex is selected, let its incident edge with color $c''$ be denoted by $e$ and the incident edge with color $c$ or $c'$ be denoted by $f$, and let this latter edge's color be denoted by $\tilde{c}$. If there are edges with color $c$ or $c'$, select one of them uniformly. Take the maximal alternating path with colors $c''$ and $\tilde{c}$ that contains $e$ but does not contain $f$. Flip the colors in this path. We denote this subcase by \emph{iv)} for reference.
    
    Eliminate all deficiency by flipping the colors on appropriate alternating path, selecting uniformly from the two neighbor edges with the same color incident to deficient vertices.
\end{enumerate}
\end{dfn}

First we prove that $M(G,k)$  is irreducible and has a small diameter. That is, the perturbations presented in it are sufficient to transform any edge $k$-coloring $C_1$ to another edge $k$-coloring $C_2$ of a bipartite graph $G = (V,E)$ in at most $6|E|$ steps.  Our strategy is the following:
\begin{enumerate}
    \item Fix an order of colors appearing in $C_2$: $c_1, c_2,\ldots, c_l$. We will have that $C_1$ is transformed via milestone realizations $C_1 = K_0, K_1, K_2,\ldots, K_{l-1} = C_2$ such that for each $i = 1, 2, \ldots l-2$ and each $j\le i$, $\restr{G}{K_i,c_j} = \restr{G}{C_2,c_j}$. Furthermore, when $K_i$ is transformed into $K_{i+1}$, no edge with color $c_1, c_2, \ldots, c_i$ changes color. These $K_i$ are called \emph{large milestones}.
    \item For each $i=1, 2,\ldots l-2$ we consider $H_i := \restr{G}{K_{i-1},c_i} \bigoplus \restr{G}{C_2,c_i}$ where $\bigoplus$ denotes the symmetric difference. The maximal degree in $H_i$ is $2$, therefore $H_i$ can be decomposed into isolated vertices, paths and cycles. Let the non-trivial components of $H_i$ be ordered and denoted by $N_1, N_2, \ldots, N_m$. $K_{i-1}$ is transformed into $K_{i}$ via milestones $K_{i-1} = L_0, L_1, \ldots L_m = K_i$ such that for all $j$, $L_j$ contains color $c_i$ only in $H_i$, and $(\restr{G}{L_j,c_i} \bigoplus \restr{G}{K_i,c_i}) \cap (N_1 \cup N_2 \cup \ldots \cup N_j)$ is the empty graph. That is, $L_j$ and $K_i$ agrees in color $c_i$ on the first $j$ non-trivial components of $H_i$. These $L_i$ are called \emph{small milestones}. The last transition from $K_{l-2}$ to $K_{l-1} = C_2$ is handled in a separate way.
    \item For each $j = 1, 2, \ldots,m$, the transformation from $L_{j-1}$ to $L_{j}$ goes in the following way. In the description of transformations, we use almost edge colorings (see Definition~\ref{def:edge-coloring}), in which at most two vertices are incident to two edges with the same color. First we give a transformation via such almost edge colorings $A_1, A_2, \cdots$ using transformations $f_1, f_2, \cdots$. Then we show how to transform these almost edge colorings to edge colorings $X_1, X_2, \cdots$ using transformations $\varphi_1, \varphi_2, \cdots$. The transformation between edge colorings $X_t$ and $X_{t+1}$ is the composition of transformations $\varphi_{t+1}\circ f_{t+1} \circ \varphi_t^{-1}$ (see also Figures~\ref{fig:between_small_milestones}~and~\ref{fig:small-milestone-example}).
    \item Transforming $K_{l-2}$ to $C_2$ is done in the following way. Consider $H = \restr{G}{K_{l-2},c_{l-1}} \bigoplus \restr{G}{C_2,c_{l-1}}$. The non-trivial components of $H$ are paths and cycles, each of which is colored alternately with colors $c_{l-2}$ and non-$c_{l-2}$ in $K_{l-2}$. For each component, we change the non-$c_{l-2}$ colors to $c_{l-1}$ and then flip the colors of the edges. This transforms $K_{l-2}$ to $C_2$.

\end{enumerate}

\begin{figure}
\begin{centering}
\begin{tikzcd}
\qquad \quad L_{j-1} \arrow{r}{f_1} \arrow[dr, bend right, dashed]
& A_1 \arrow{r}{f_2} \arrow[shift right, swap]{d}{\varphi_1}
& A_2  \arrow[shift right, swap]{d}{\varphi_2} \arrow{r}{f_3} &\cdots \cdots \arrow{r}{f_{e-2}} &A_{e-2} \arrow{r}{f_{e-1}} \arrow[shift right, swap]{d}{\varphi_{e-2}} &A_{e-1} \arrow{r}{f_e} \arrow[shift right, swap]{d}{\varphi_{e-1}} &L_{j}\\
& X_1\arrow[shift right, swap]{u}{\varphi_1^{-1}} \arrow[r, dashed] &X_2\arrow[shift right, swap]{u}{\varphi_2^{-1}}\arrow[r, dashed] &\cdots \cdots \arrow[r, dashed] &X_{e-2} \arrow[r, dashed] \arrow[shift right, swap]{u}{\varphi_{e-2}^{-1}} &X_{e-1} \arrow[ur, bend right, dashed] \arrow[shift right, swap]{u}{\varphi_{e-1}^{-1}}
\end{tikzcd}
\end{centering}

    \caption{The transformation between two small milestones. See text for details.}
    \label{fig:between_small_milestones}
\end{figure}
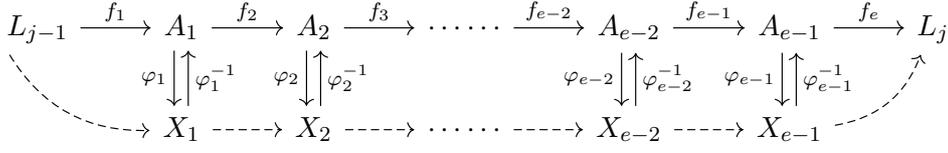

\begin{figure}
    \centering
      \begin{adjustbox}{max totalsize={\textwidth}{0.9\textheight},center}
\begin{tikzpicture}
\definecolor{dark green}{rgb}{0,0.7,0}
\definecolor{dark blue}{rgb}{0,0,0.7}

\node[] at (-1,1.5) {$L$};
\node[] at (1.5,1.5) {$N$};
\node[] at (2.3,1.5) {$e_1$};
\node[] at (1.5,0.7) {$e_2$};
\node[] at (0.7,1.5) {$e_3$};
\node[] at (1.5,2.3) {$e_4$};

\draw [black,fill=black] (0,0) circle (0.05cm);
\draw [black,fill=black] (0,3) circle (0.05cm);
\draw [black,fill=black] (3,0) circle (0.05cm);
\draw [black,fill=black] (3,3) circle (0.05cm);
\draw [black,fill=black] (1,1) circle (0.05cm);
\draw [black,fill=black] (1,2) circle (0.05cm);
\draw [black,fill=black] (2,1) circle (0.05cm);
\draw [black,fill=black] (2,2) circle (0.05cm);
\draw [ultra thick,red] (1,1) to  (2,1);
\draw [ultra thick,red] (1,2) to  (2,2);
\draw [ultra thick,red] (0,0) to  (3,0);
\draw [ultra thick,red] (0,3) to  (3,3);
\draw [dark green] (1,1) to  (1,2);
\draw [dark green] (0,0) to  (0,3);
\draw [dark green] (2,1) to  (3,0);
\draw [dark green] (2,2) to  (3,3);
\draw [dotted,dark blue] (2,1) to  (2,2);
\draw [dotted,dark blue] (3,0) to  (3,3);
\draw [dotted,dark blue] (0,0) to  (1,1);
\draw [dotted,dark blue] (1,2) to  (0,3);

\draw [->,>=triangle 60] (1.5,3.5) to  (1.5,4.5);
\node[] at (2,4) {$f_1$};

\draw [->,>=triangle 60] (3.5,1.5) to  (11.5,4.5);
\node[] at (8,2.7) {$\varphi_1\circ f_1$};

\node[] at (-1,6.5) {$A_1$};
\draw [black,fill=black] (0,5) circle (0.05cm);
\draw [black,fill=black] (0,8) circle (0.05cm);
\draw [black,fill=black] (3,5) circle (0.05cm);
\draw [black,fill=black] (3,8) circle (0.05cm);
\draw [black,fill=black] (1,6) circle (0.05cm);
\draw [black,fill=black] (1,7) circle (0.05cm);
\draw [black,fill=black] (2,6) circle (0.05cm);
\draw [black,fill=black] (2,7) circle (0.05cm);
\draw [ultra thick,red] (1,6) to  (2,6);
\draw [ultra thick,red] (1,7) to  (2,7);
\draw [ultra thick,red] (0,5) to  (3,5);
\draw [ultra thick,red] (0,8) to  (3,8);
\draw [dark green] (1,6) to  (1,7);
\draw [dark green] (0,5) to  (0,8);
\draw [dark green] (2,6) to  (3,5);
\draw [dark green] (2,7) to  (3,8);
\draw [ultra thick,red] (2,6) to  (2,7);
\draw [dotted,dark blue] (3,5) to  (3,8);
\draw [dotted,dark blue] (0,5) to  (1,6);
\draw [dotted,dark blue] (1,7) to  (0,8);

\draw [->,>=triangle 60] (3.5,6.5) to  (9.5,6.5);
\node[] at (7,6) {$\varphi_1$};

\node[] at (14,6.5) {$X_1$};
\draw [black,fill=black] (10,5) circle (0.05cm);
\draw [black,fill=black] (10,8) circle (0.05cm);
\draw [black,fill=black] (13,5) circle (0.05cm);
\draw [black,fill=black] (13,8) circle (0.05cm);
\draw [black,fill=black] (11,6) circle (0.05cm);
\draw [black,fill=black] (11,7) circle (0.05cm);
\draw [black,fill=black] (12,6) circle (0.05cm);
\draw [black,fill=black] (12,7) circle (0.05cm);
\draw [dotted,dark blue] (11,6) to  (12,6);
\draw [dotted,dark blue] (11,7) to  (12,7);
\draw [dotted,dark blue] (10,5) to  (13,5);
\draw [dotted,dark blue] (10,8) to  (13,8);
\draw [dark green] (11,6) to  (11,7);
\draw [dark green] (10,5) to  (10,8);
\draw [dark green] (12,6) to  (13,5);
\draw [dark green] (12,7) to  (13,8);
\draw [ultra thick, red] (12,6) to  (12,7);
\draw [ultra thick, red] (13,5) to  (13,8);
\draw [ultra thick, red] (10,5) to  (11,6);
\draw [ultra thick, red] (11,7) to  (10,8);

\node[] at (-1,11.5) {$A_2$};
\draw [black,fill=black] (0,10) circle (0.05cm);
\draw [black,fill=black] (0,13) circle (0.05cm);
\draw [black,fill=black] (3,10) circle (0.05cm);
\draw [black,fill=black] (3,13) circle (0.05cm);
\draw [black,fill=black] (1,11) circle (0.05cm);
\draw [black,fill=black] (1,12) circle (0.05cm);
\draw [black,fill=black] (2,11) circle (0.05cm);
\draw [black,fill=black] (2,12) circle (0.05cm);
\draw [dotted, dark blue] (1,11) to  (2,11);
\draw [ultra thick,red] (1,12) to  (2,12);
\draw [ultra thick,red] (0,10) to  (3,10);
\draw [ultra thick,red] (0,13) to  (3,13);
\draw [dark green] (1,11) to  (1,12);
\draw [dark green] (0,10) to  (0,13);
\draw [dark green] (2,11) to  (3,10);
\draw [dark green] (2,12) to  (3,13);
\draw [ultra thick, red] (2,11) to  (2,12);
\draw [dotted,dark blue] (3,10) to  (3,13);
\draw [dotted,dark blue] (0,10) to  (1,11);
\draw [dotted,dark blue] (1,12) to  (0,13);

\node[] at (14,11.5) {$X_2$};
\draw [black,fill=black] (10,10) circle (0.05cm);
\draw [black,fill=black] (10,13) circle (0.05cm);
\draw [black,fill=black] (13,10) circle (0.05cm);
\draw [black,fill=black] (13,13) circle (0.05cm);
\draw [black,fill=black] (11,11) circle (0.05cm);
\draw [black,fill=black] (11,12) circle (0.05cm);
\draw [black,fill=black] (12,11) circle (0.05cm);
\draw [black,fill=black] (12,12) circle (0.05cm);
\draw [dotted,dark blue] (11,11) to  (12,11);
\draw [dotted,dark blue] (11,12) to  (12,12);
\draw [dotted,dark blue] (10,10) to  (13,10);
\draw [dotted,dark blue] (10,13) to  (13,13);
\draw [dark green] (11,11) to  (11,12);
\draw [dark green] (10,10) to  (10,13);
\draw [dark green] (12,11) to  (13,10);
\draw [dark green] (12,12) to  (13,13);
\draw [ultra thick, red] (12,11) to  (12,12);
\draw [ultra thick, red] (13,10) to  (13,13);
\draw [ultra thick, red] (10,10) to  (11,11);
\draw [ultra thick, red] (11,12) to  (10,13);

\draw [->,>=triangle 60] (1.5,8.5) to  (1.5,9.5);
\node[] at (2,9) {$f_2$};

\draw [->,>=triangle 60] (3.5,11.5) to  (9.5,11.5);
\node[] at (7,11) {$\varphi_2$};

\draw [->,>=triangle 60] (11.5,8.5) to  (11.5,9.5);
\node[] at (13,9) {$\varphi_2 \circ f_2 \circ \varphi_1^{-1}$};

\node[] at (-1,16.5) {$A_3$};
\draw [black,fill=black] (0,15) circle (0.05cm);
\draw [black,fill=black] (0,18) circle (0.05cm);
\draw [black,fill=black] (3,15) circle (0.05cm);
\draw [black,fill=black] (3,18) circle (0.05cm);
\draw [black,fill=black] (1,16) circle (0.05cm);
\draw [black,fill=black] (1,17) circle (0.05cm);
\draw [black,fill=black] (2,16) circle (0.05cm);
\draw [black,fill=black] (2,17) circle (0.05cm);
\draw [dotted, dark blue] (1,16) to  (2,16);
\draw [ultra thick,red] (1,17) to  (2,17);
\draw [ultra thick,red] (0,15) to  (3,15);
\draw [ultra thick,red] (0,18) to  (3,18);
\draw [dotted,dark blue] (1,16) to  (1,17);
\draw [dotted,dark blue] (0,15) to  (0,18);
\draw [dark green] (2,16) to  (3,15);
\draw [dark green] (2,17) to  (3,18);
\draw [ultra thick, red] (2,16) to  (2,17);
\draw [dotted,dark blue] (3,15) to  (3,18);
\draw [dark green] (0,15) to  (1,16);
\draw [dark green] (1,17) to  (0,18);

\node[] at (14,16.5) {$X_3$};
\draw [black,fill=black] (10,15) circle (0.05cm);
\draw [black,fill=black] (10,18) circle (0.05cm);
\draw [black,fill=black] (13,15) circle (0.05cm);
\draw [black,fill=black] (13,18) circle (0.05cm);
\draw [black,fill=black] (11,16) circle (0.05cm);
\draw [black,fill=black] (11,17) circle (0.05cm);
\draw [black,fill=black] (12,16) circle (0.05cm);
\draw [black,fill=black] (12,17) circle (0.05cm);
\draw [ultra thick,red] (11,16) to  (12,16);
\draw [ultra thick,red] (11,17) to  (12,17);
\draw [ultra thick,red] (10,15) to  (13,15);
\draw [ultra thick,red] (10,18) to  (13,18);
\draw [dotted,dark blue] (11,16) to  (11,17);
\draw [dotted,dark blue] (10,15) to  (10,18);
\draw [dark green] (12,16) to  (13,15);
\draw [dark green] (12,17) to  (13,18);
\draw [dotted,dark blue] (12,16) to  (12,17);
\draw [dotted,dark blue] (13,15) to  (13,18);
\draw [dark green] (10,15) to  (11,16);
\draw [dark green] (11,17) to  (10,18);

\draw [->,>=triangle 60] (1.5,13.5) to  (1.5,14.5);
\node[] at (2,14) {$f_3$};

\draw [->,>=triangle 60] (3.5,16.5) to  (9.5,16.5);
\node[] at (7,16) {$\varphi_3$};

\draw [->,>=triangle 60] (11.5,13.5) to  (11.5,14.5);
\node[] at (13,14) {$\varphi_3 \circ f_3 \circ \varphi_2^{-1}$};

\node[] at (-1,21.5) {$A_4$};

\draw [black,fill=black] (0,20) circle (0.05cm);
\draw [black,fill=black] (0,23) circle (0.05cm);
\draw [black,fill=black] (3,20) circle (0.05cm);
\draw [black,fill=black] (3,23) circle (0.05cm);
\draw [black,fill=black] (1,21) circle (0.05cm);
\draw [black,fill=black] (1,22) circle (0.05cm);
\draw [black,fill=black] (2,21) circle (0.05cm);
\draw [black,fill=black] (2,22) circle (0.05cm);
\draw [dotted, dark blue] (1,21) to  (2,21);
\draw [ultra thick,red] (1,22) to  (2,22);
\draw [ultra thick,red] (0,20) to  (3,20);
\draw [ultra thick,red] (0,23) to  (3,23);
\draw [ultra thick,red] (1,21) to  (1,22);
\draw [dotted,dark blue] (0,20) to  (0,23);
\draw [dark green] (2,21) to  (3,20);
\draw [dark green] (2,22) to  (3,23);
\draw [ultra thick, red] (2,21) to  (2,22);
\draw [dotted,dark blue] (3,20) to  (3,23);
\draw [dark green] (0,20) to  (1,21);
\draw [dark green] (1,22) to  (0,23);

\node[] at (14,21.5) {$X_4$};
\draw [black,fill=black] (10,20) circle (0.05cm);
\draw [black,fill=black] (10,23) circle (0.05cm);
\draw [black,fill=black] (13,20) circle (0.05cm);
\draw [black,fill=black] (13,23) circle (0.05cm);
\draw [black,fill=black] (11,21) circle (0.05cm);
\draw [black,fill=black] (11,22) circle (0.05cm);
\draw [black,fill=black] (12,21) circle (0.05cm);
\draw [black,fill=black] (12,22) circle (0.05cm);
\draw [dotted,dark blue] (11,21) to  (12,21);
\draw [dotted,dark blue] (11,22) to  (12,22);
\draw [ultra thick,red] (10,20) to  (13,20);
\draw [ultra thick,red] (10,23) to  (13,23);
\draw [ultra thick,red] (11,21) to  (11,22);
\draw [dotted,dark blue] (10,20) to  (10,23);
\draw [dark green] (12,21) to  (13,20);
\draw [dark green] (12,22) to  (13,23);
\draw [ultra thick,red] (12,21) to  (12,22);
\draw [dotted,dark blue] (13,20) to  (13,23);
\draw [dark green] (10,20) to  (11,21);
\draw [dark green] (11,22) to  (10,23);

\draw [->,>=triangle 60] (1.5,18.5) to  (1.5,19.5);
\node[] at (2,19) {$f_4$};

\draw [->,>=triangle 60] (3.5,21.5) to  (9.5,21.5);
\node[] at (7,21) {$\varphi_4$};

\draw [->,>=triangle 60] (11.5,18.5) to  (11.5,19.5);
\node[] at (13,19) {$\varphi_4 \circ f_4 \circ \varphi_3^{-1}$};

\draw [->,>=triangle 60] (1.5,23.5) to  (1.5,24.5);
\node[] at (2,24) {$f_5$};

\draw [->,>=triangle 60] (11.5,23.5) to  (3.5,26.5);
\node[] at (7.5,26) {$f_5 \circ \varphi_4^{-1}$};

\node[] at (-1,26.5) {$L'$};
\draw [black,fill=black] (0,25) circle (0.05cm);
\draw [black,fill=black] (0,28) circle (0.05cm);
\draw [black,fill=black] (3,25) circle (0.05cm);
\draw [black,fill=black] (3,28) circle (0.05cm);
\draw [black,fill=black] (1,26) circle (0.05cm);
\draw [black,fill=black] (1,27) circle (0.05cm);
\draw [black,fill=black] (2,26) circle (0.05cm);
\draw [black,fill=black] (2,27) circle (0.05cm);
\draw [dotted, dark blue] (1,26) to  (2,26);
\draw [dotted,dark blue] (1,27) to  (2,27);
\draw [ultra thick,red] (0,25) to  (3,25);
\draw [ultra thick,red] (0,28) to  (3,28);
\draw [ultra thick, red] (1,26) to  (1,27);
\draw [dotted,dark blue] (0,25) to  (0,28);
\draw [dark green] (2,26) to  (3,25);
\draw [dark green] (2,27) to  (3,28);
\draw [ultra thick, red] (2,26) to  (2,27);
\draw [dotted,dark blue] (3,25) to  (3,28);
\draw [dark green] (0,25) to  (1,26);
\draw [dark green] (1,27) to  (0,28);



\end{tikzpicture}%
\end{adjustbox}
   \caption{Transformation of the edge coloring $L$ to the edge coloring $L'$. See text for details.}
    \label{fig:small-milestone-example}
\end{figure}
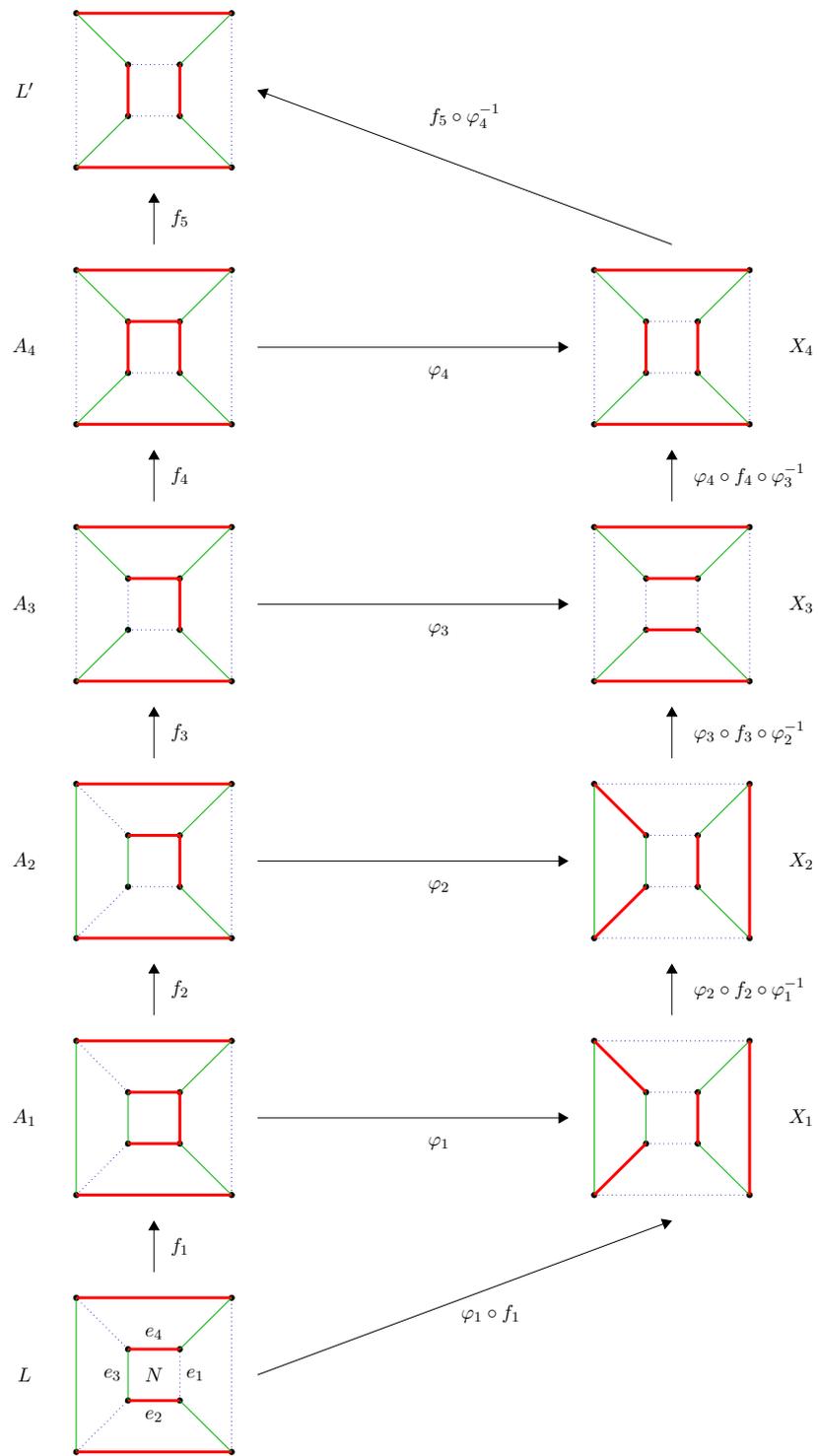

The precise description of the transformations from $L_{j-1}$ to $L_j$ is given by the following lemma and by Lemma~\ref{lem:transf-from-almost-to-proper}.
\begin{lem}\label{lem:path-between-small-milestones}
Let $L$ be an edge coloring of a bipartite graph $G$ with $k \ge 3$ colors, and let $N$ be a connected subgraph of $G$ with maximal degree $2$ and $s$ edges.
Assume that there exists a color $c$ such that the colors of the edges in $N$ are alternately $c$ and non-$c$. We further assume that for a path $N$, if any of its end edges has a non-$c$ color then $N$ cannot be extended with a $c$ edge, and if any of the end edges has color $c$ then that vertex has degree less than $k$. 
Then $L$ can be transformed into an edge coloring $L'$ in at most $\left\lceil\frac{3}{2}s\right\rceil$ steps such that $\restr{G}{L,c} \bigoplus \restr{G}{L',c} = N$ with the condition below: there exists a color $c'$ satisfying that at each step, one of the following transformation is performed:
\begin{enumerate}
    \item Changing the color of an edge from $c$ to some $c'$.
    \item Changing the color of an edge from some $c'$ to $c$.
    \item In a maximal path or a cycle containing edges alternately colored by $c'$ and $c''$  for some colors $c',c''\ne c$, flipping the color of each edge from $c'$ to $c''$ and vice versa.
\end{enumerate}
Furthermore, at each step, the so-obtained perturbation gives an almost edge coloring.
\end{lem}
Before we prove Lemma~\ref{lem:path-between-small-milestones}, we give an example illustrated on the left hand side in Figure~\ref{fig:small-milestone-example}. In this example, an edge coloring $L$ is transformed into an edge coloring $L'$ via almost edge colorings in five steps. The component $N$ is a $4$-cycle of alternating red and non-red edges, labeled by $e_1$, $e_2$, $e_3$, and $e_4$. First the color of $e_1$ is changed from blue to red. Then the color of $e_2$ is changed from red to blue. Then colors of the edges in the alternating cycle of blue and green edges containing $e_3$ are flipped. Then the color of $e_3$ is changed from blue to red. Finally the color of $e_4$ is changed from red to blue.

\begin{proof}[Proof of Lemma~\ref{lem:path-between-small-milestones}]
$N$ is either a path or a cycle. In both cases, the colors of its edges alternate between $c$ and non-$c$ in $L$. Order the vertices of $N$ by traveling around a cycle or from one to the another end for a path. Start the walk along the cycle by a non-$c$ edge or from the end with a non-$c$ edge if such end exists in the path.  In both cases, let color of the non-$c$ edge be the $c'$ stated in the lemma. Otherwise start with a $c$ edge. In this case, let $c'$ be a color that is not among the colors of those edges incident to the end vertex. Let the vertices along the walk be denoted by $v_1, v_2,\ldots, v_s$. 

These transformations are applied in the following way: If $(v_1,v_2)$ has color $c'$, then the first step is to change it to $c$. In all other cases, the first step is to change the color of $(v_1,v_2)$ from $c$ to $c'$. Then we consider the edges along the walk one by one, and change their colors in one step if the color of the current edge is $c$ and in at most two steps otherwise. If the color of the current edge is $c$, the next step is to change it to $c'$. If the color of the current edge is $c'$, then the next step is to change it to $c$. If the color of the current edge $e$ is some $c''\ne c,c'$, then first we consider the maximal alternating path or alternating cycle with colors $c'$ and $c''$ containing $e$ and not containing the previous edge in $N$ and flip the colors of the edges of this path or cycle. After this step, $e$ has color $c'$, so the next step is to change this color to $c$. 

We claim that along these transformations, we have almost edge $k$-colorings for some colors $c$ and $c'$.

Indeed, if $N$ is a cycle, then the first transformation creates two $(+c-c')$-deficient vertices. If $N$ is a path, and its first edge has color $c'$, then the first transformation creates a $(+c-c')$-deficient vertex. Indeed, observe that $v_1$ is not incident to a $c$ vertex by the condition in the lemma. If $N_j$ is a path, and its first edge has color $c$, then the first transformation creates at most one $(+c'-c)$-deficient vertex (in case when $v_2$ has an incident $c'$ edge before the transformation).

If the current edge $e=(v_l, v_{l+1})$ has color $c$, then vertex $v_l$ has a $(+c-c')$ deficiency. By changing the color of $e$ from $c$ to $c'$, we eliminate this deficiency. However, we might create a new, $(+c'-c)$-deficiency for vertex $v_{l+1}$ if it has an incident $c'$ edge before the transformation. If the edge $(v_{l+1},v_{l+2})$ does not have color $c'$, then its color is some $c''$. Find the longest alternating path or cycle with colors $c'$ and $c''$ containing $(v_{l+1}, v_{l+2})$ but not containing $(v_{l},v_{l+1})$ and flip the colors in it. After this step, $v_{l+1}$ has a $(+c'-c)$-deficiency. Now comes changing the color of the edge $(v_{l+1},v_{l+2})$ from $c'$ to $c$, that eliminates the deficiency of $v_{l+1}$ and creates a $(+c-c')$ deficiency for $v_{l+2}$.

This eliminating a deficiency and creating a new one continues until we get back to $v_1$ in the cycle or to the other end of the path. 

If $N$ is a cycle, in the last step or possibly in the last but one step we change the color of $(v_s,v_1)$ from $c$ to $c'$. This eliminates the $(+c-c')$-deficiency of both vertex $v_s$ and vertex $v_1$ if $v_1$ still has deficiency $(+c-c')$. If $v_1$ changed the type of deficiency from $(+c-c')$ to some $(+c-c'')$ because in one of the steps, a maximal alternating path with colors $c'$ and $c''$ ended in $v_1$, then flipping the color of $(v_s,v_1)$ from $c$ to $c'$ changes the type of deficiency of $v_1$ from $(+c-c'')$ to $(+c'-c'')$. Then consider a longest alternating path with colors $c'$ and $c''$ starting in $v_1$, and flip the colors in it. This eliminates the deficiency of $v_1$.

If $N$ is a path, then the last step is to change the color of $(v_{s-1},v_s)$ from $c'$ to $c$ or possibly the last but one step is to change the color of $(v_{s-1},v_s)$ from $c$ to $c'$. This eliminates the deficiency of $v_{s-1}$ and possibly creates a new $(+c'-c)$-deficiency for $v_s$, if it has an incident $c'$ edge before the step in question. In that case, there exists a $c''$ color, such that $v_s$ is not incident to any edge with $c''$ color. Consider the longest alternating path with colors  $c'$ and $c''$ containing edge $(v_{s-1},v_s)$. Such an alternating path exists since $G$ is a bipartite graph. Indeed, an alternating path could return to $v_s$ after even number of steps, which would mean an edge with color $c''$. However, $v_s$ does not incident to an edge with color $c''$. Flip the colors in this path; this eliminates the deficiency of $v_s$.

Every second edge can be transformed in one step, and every other edge can be transformed in at most two steps, furthermore, the first edge can be transformed in one step, and the last in at most two steps, therefore the number of steps is indeed at most $\left\lceil\frac{3}{2}s\right\rceil$ steps.
\end{proof}


$N_j$ is a non-trivial component of $H_i$, and thus either it is a path or a cycle. In both cases, the color of its edges alternate between $c_j$ and non-$c_j$ in $L_{j-1}$. If it is a path and any of the end edges has a non-$c_j$ color, then it cannot be extended by a $c_j$ edge. If any of the end edges has color $c_j$, then its end vertex is not incident to a $c_j$ edge in $C_2$, thus its degree is smaller than $k$. Furthermore, the number of colors is at least $3$, therefore, Lemma~\ref{lem:path-between-small-milestones} can be applied to transform $L_{j-1}$ to $L_j$ by setting $N$ to $N_j$ and the color $c$ to $c_j$.



We are ready to state and prove the theorem on irreducibility and small diameter.

\begin{thm}
Let $C_1$ and $C_2$ be two edge $k$-colorings of the same bipartite graph $G = (V,E)$. Then $C_1$ can be transformed into $C_2$ in at most $6|E|$ steps in the Markov chain $M(G,k)$.
\end{thm}
\begin{proof}
First we show that these transformations are sufficient, then we prove the upper bound on the number of necessary transformations.

Observe that from $K_0$ to $K_{l-2}$, any large milestone is also a small milestone. Therefore it is enough to show that the listed perturbations are sufficient to transform one small milestone to the next small milestone and they are sufficient to transform $K_{l-2}$ to $K_{l-1}$. Observe that the listed perturbations contains the case when the colors are flipped along a maximal alternating path or cycle. Indeed, in case (2)(b) in Definition~\ref{def:markov-chain}, a maximal sub-path or a cycle can also be selected. These transformations are sufficient to transform $K_{l-2}$ to $K_{l-1}$.

Any step between two small milestones is a transformation in the form $\varphi_{t+1}\circ f_{t+1}\circ \varphi_t^{-1}$ or $\varphi_1\circ f_1$ or $f_e\circ \varphi_{e-1}^{-1}$, where  $\varphi$ is a transformation given in Lemma~\ref{lem:transf-from-almost-to-proper} and $f$ is a transformation given in Lemma~\ref{lem:path-between-small-milestones} (see also Figures~\ref{fig:between_small_milestones}~and~\ref{fig:small-milestone-example}.). A transformation $\varphi$ corrects the deficiency of at most two vertices by flipping the edges along at most two maximal alternating paths. Therefore, its inverse creates at most two deficiencies along at most two alternating paths, and these deficiencies appear at the end of the path that could be extended. With some probability, no deficiency is created, thus providing the  transformations of type $\varphi_1\circ f_1$.

A transformation $f$ achieves the following:
\begin{enumerate}[label=\roman*)]
    \item  changes the color of an edge which is adjacent to an edge with the same color. This happens when the color of an edge is changed from $c$ to $c'$ or from $c'$ to $c$, and the previous edge in component $N$ described in the proof of Lemma~\ref{lem:path-between-small-milestones} has also color $c$ or $c'$, or
    \item  changes the color of an arbitrary edge if there is no deficient vertex. This happens in the first step of transforming component $N$ as described in the proof of Lemma~\ref{lem:path-between-small-milestones}, or
    \item flips the colors along a maximal alternating path or cycle in which one of the vertices is deficient with one of the colors in the alternating path or cycle. This happens when the current edge $e$ (as described in the proof of Lemma~\ref{lem:path-between-small-milestones}) has color $c''$ and is adjacent to two edges with colors $c'$, or
    \item flips the colors along an alternating path thus creating exactly one deficient vertex. This happens when the current edge $e$ (as described in the proof of Lemma~\ref{lem:path-between-small-milestones}) has color $c''$, and the only adjacent edge with color $c'$ is the previous edge in component $N$.
\end{enumerate}
These are exactly the four subcases given in Definition~\ref{def:markov-chain}.

Finally, the transformation $\varphi_{t+1}$ eliminates all the deficiencies, if such deficiencies exist.

In each subgraph $H_i$, in each non-trivial component at least one third of the edges have color $c_i$ in $C_2$. Therefore, the total number of edges of the non-trivial components in all $H_i$'s is at most $3|E|$. Since for each component the number of necessary transformations is at most $\left\lceil\frac{3}{2}s\right\rceil \le 2s$, where $s$ is the number of edges in the given component, $6|E|$ number of steps are sufficient to transform any edge coloring to any other one.
\end{proof}

As an example, we show how to generate the transformations in Figure~\ref{fig:small-milestone-example} by the transformations in Definition~\ref{def:markov-chain}.
\begin{enumerate}
    \item $\varphi_1 \circ f_1$: We select two colors, red and blue, then do nothing (case (2)(a)). Since there is no deficient vertex, we can change the color of edge $e_1$ from blue to red. Then we eliminate all deficiencies by flipping the colors along a maximal alternating path between the two deficient vertices emerged. We can conclude that $P(\varphi_1\circ f_1) =  \frac{1}{4}\times \frac{1}{3} \times \frac{1}{3} \times \frac{1}{2} \times \frac{1}{8} \times \frac{1}{2}$. Indeed, $\frac{1}{4}$ is the probability of case (2), there are ${3\choose 2} = 3$ possible pair of colors, $\frac{1}{3}$ is the probability of subcase (a) in case (2), $\frac{1}{2}$ is the probability of changing the color of an edge when there is no deficient vertex in subcase (2)(b), and there are $8$ edges with color blue or red, and finally, we select the path in which $A_1$ and $X_1$ differ with probability $\frac{1}{2}$.
    \item $\varphi_2\circ f_2\circ \varphi_1^{-1}$: We select two colors, red and blue, then we select the subpath in which $X_1$ and $A_1$ differ (case (2)(b)). We flip the colors along this path. Since there are two deficient vertices incident to edge $e_1$, we select the vertex incident to both $e_1$ and $e_2$, then then we select the edge $e_2$, and flip its color from red to blue. Finally, we eliminate all deficiencies by flipping the colors along a maximal alternating path. We can conclude that $P(\varphi_2 \circ f_2 \circ \varphi_1^{-1}) = \frac{1}{2} \times \frac{1}{3} \times \frac{1}{3} \times \frac{1}{57} \times \frac{1}{4} \times \frac{1}{2}$. Indeed, $\frac{1}{4}$ is the probability of case (2), there are ${3\choose 2} = 3$ possible pair of colors, $\frac{1}{3}$ is the probability of subcase (b) in case (2), there are ${8\choose 2} \times 2$ subpaths in the subgraph with color blue and red and $1$ cycle, thus we have to select one of the $57$ possibilities, selecting edge $e_2$ has probability $\frac{1}{4}$ (selecting uniformly one of the deficient vertices and uniformly one of the edges with repeated color), and finally, we select the path in which $A_2$ and $X_2$ differ with probability $\frac{1}{2}$.
    \item $\varphi_3 \circ f_3 \circ \varphi_2^{-1}$: We select three colors, red, blue and green, and we select green as the distinguished color $c''$. Then we select the path in which $X_2$ and $A_2$ differ, and swap the red and blue colors. We select the alternating blue-green cycle containing the edge $e_3$ and flip the green and blue colors in it. Then we eliminate the deficiencies by flipping the colors along the path with which $A_3$ and $X_3$ differ. We can conclude that $P(\varphi_3\circ f_2\varphi_2^{-1}) = \frac{1}{4}\times \frac{1}{3} \times \frac{1}{3}\times \frac{1}{57} \times \frac{1}{4} \times \frac{1}{2}$. Indeed, the probability of case (3) is $\frac{1}{4}$, there is one way to select three colors, and there are three ways to select one of them distinguished, there are ${8\choose 2} \times 2$ subpaths in the subgraph with color blue and red and $1$ cycle, thus we have to select one of the $57$ possibilities, selecting edge $e_3$ has probability $\frac{1}{4}$ (selecting uniformly one of the deficient vertices and uniformly one of the edges with repeated color), and finally, we select the path in which $A_3$ and $X_3$ differ with probability $\frac{1}{2}$.
    \item $\varphi_4 \circ f_4 \circ \varphi_3^{-1}$: We select two colors, red and blue. Then we select the path in which $X_3$ and $A_3$ differ, and flip the colors of the edges in it. Then we select $e_3$, and flip its color from blue to red. Finally, we eliminate all deficiencies by flipping the colors along a maximal alternating path. We can conclude that $P(\varphi_4 \circ f_4 \circ \varphi_3^{-1}) = \frac{1}{4} \times \frac{1}{3} \times \frac{1}{3} \times \frac{1}{26} \times \frac{1}{4} \times \frac{1}{2}$. Indeed, $\frac{1}{4}$ is the probability of case (2), there are ${3\choose 2} = 3$ possible pair of colors, $\frac{1}{3}$ is the probability of subcase (b) in case (2), there are ${4\choose 2} \times 2 \times 2 = 24$ subpaths in the subgraph with color blue and red and $2$ cycles, thus we have to select one of the $26$ possibilities, selecting edge $e_3$ has probability $\frac{1}{4}$ (selecting uniformly one of the deficient vertices and uniformly one of the edges with repeated color), and finally, we select the path in which $A_4$ and $X_4$ differ with probability $\frac{1}{2}$.
    \item $f_5\circ \varphi_4^{-1}$:  We select two colors, red and blue. Then we select the path in which $X_4$ and $A_4$ differ, and flip the colors of the edges in it. Then we select $e_4$, and flip its color from red to blue. We can conclude that $P(\varphi_4 \circ f_4 \circ \varphi_3^{-1}) = \frac{1}{4} \times \frac{1}{3} \times \frac{1}{3} \times \frac{1}{26} \times \frac{1}{4}$. Indeed, $\frac{1}{4}$ is the probability of case (2), there are ${3\choose 2} = 3$ possible pair of colors, $\frac{1}{3}$ is the probability of subcase (b) in case (2), there are ${4\choose 2} \times 2 \times 2 = 24$ subpaths in the subgraph with color blue and red and $2$ cycles, thus we have to select one of the $26$ possibilities, selecting edge $e_4$ has probability $\frac{1}{4}$ (selecting uniformly one of the deficient vertices and uniformly one of the edges with repeated color).
\end{enumerate}

We can put $M(G,k)$ into a Metropolis-Hastings algorithm to design a Markov chain Monte Carlo converging to the uniform distribution of edge $k$-colorings of a bipartite graph. We can use the variant introduced by Lunter \emph{et. al} \cite{lunteretal2005}, that is, when an edge coloring $C_2$ is proposed from $C_1$ in a way $w$, we obtain the inverse way $w'$ transforming $C_2$ back to $C_1$, and use the probabilities $P(C_1,w'|C_2)$ and $P(C_2,w,|C_1)$ in the Metropolis-Hastings ratio. Furthermore, since the target distribution $\pi$ is the uniform one, $g$ can be set to any constant function, and is cancelled in the Metropolis-Hastings ratio, which simply becomes
$$
\frac{P(C_1,w'|C_2)}{P(C_2,w|C_1)},
$$
where edge $k$-coloring $C_2$ is proposed from edge $k$-coloring  $C_1$ in way $w$.
Observe that the smallest acceptance probability is the minimum of the Metropolis-Hastings ratio, so the maximum of the inverse of the acceptance ratio is the maximum of the inverse of the Metropolis-Hastings ratio, and due to symmetry, it is simply the maximum of the Metropolis-Hastings ratio. Below we give upper bound for this maximum when $M(G,k)$ is used in the Metropolis-Hastings algorithm.

Since for all $C_1$ and $C_2$ and possible transformation way $w$ between them, $P(C_1,w'|C_2)<1$, the maximum of the Metropolis-Hastings ratio is at most 
$$
\frac{1}{P(C_2,w|C_1)}.
$$
This already would give a polynomial upper bound for the inverse of the acceptance ratio, however, with a more careful analysis, a better upper bound could be found.

If nothing is done in the Markov chain, then the Metropolis-Hastings ratio is $1$. Otherwise, a step in the Markov chain does the following:
\begin{enumerate}
    \item It selects either $2$ or $3$ colors.
    \item Based on the selected colors, it creates at most two deficient vertices. We denote this transformation $\varphi_1^{-1}$.
    \item It perturbs the current configuration. We call this perturbation $f$.
    \item Eliminates all the deficiencies. We denote this transformation $\varphi_2$.
\end{enumerate}
Observe that in the reverse transformation, the same colors must be selected, therefore, the probability of selecting the particular colors cancels in the Metropolis-Hastings ratio, thus it simplifies to
$$
\frac{P(\varphi_2^{-1})P(f^{-1})P(\varphi_1)}{P(\varphi_1^{-1})P(f)P(\varphi_2)}
$$

We give individual upper bounds on $\frac{P(\varphi_2^{-1})}{P(\varphi_2)}$, $\frac{P(f^-1)}{P(f)}$ and $\frac{P(\varphi_1)}{P(\varphi_1^{-1})}$, and their product is an upper bound on the Metropolis-Hastings ratio.

We claim that
$$
\frac{P(\varphi_2^{-1})}{P(\varphi_2)} \le \frac{1}{P(\varphi_2)} \le 4.
$$
Indeed, any probability is smaller than $1$, thus the first inequality holds. $\varphi_2$ eliminates at most two deficiencies. It must select uniformly one of the edges with the same color incident to a deficient vertex, and it has probability $\frac{1}{2}$. Once the edge is selected, the maximal alternating path with the two prescribed colors is defined unequivocally, and thus with probability $1$. Since at most twice an edge incident to a deficient vertex must be selected during the procedure $\varphi_2$, its probability has a minimum $\frac{1}{4}$, and thus the second inequality.

Now we show that
$$
\frac{P(\varphi_1)}{P(\varphi_1^{-1})} \le \frac{1}{P(\varphi_1^{-1})} \le 12{|V|\choose 2}.
$$
The reasoning for the first inequality is the same as above. $\varphi_1^{-1}$ creates at most two deficient vertices by selecting at most two alternating sub-paths in the subgraph $H$ defined by two colors, and then flipping the colors along these paths. If one path is selected, then this option is chosen with $\frac{1}{3}$ probability. After this,  a sub-path is selected uniformly. A sub-path is defined by its end vertices, therefore there cannot be more than ${|V|\choose 2}$ sub-paths, so the probability of selecting one sub-path is at least $\frac{1}{{|V|\choose 2}}$. If two sub-paths are selected, then this option is chosen with $\frac{1}{3}$ probability. In this case, both of the sub-paths must have an end vertex which is an end vertex also in $H$. Thus only the other end points must be selected, and there are $2$ options for this. Therefore, the number of pair of sub-paths with the required properties is at most $4{|V|\choose 2}$, so the probability of selecting one of them is at least $\frac{1}{4{|V|\choose 2}}$. Therefore, it indeed holds that
$$
\frac{1}{P(\varphi_1^{-1})}\le \frac{1}{\frac{1}{3}\times \min\left\{\frac{1}{{|V|\choose 2}}, \frac{1}{4{|V|\choose 2}}\right\}} = 12 {|V|\choose 2}. 
$$

Finally, we show that
$$
\frac{P(f^{-1})}{P(f)} \le \frac{1}{P(f)} \le 4|V|.
$$
The reasoning for the first inequality is the same as above. The $f$ transformation does one of the following:
\begin{enumerate}
    \item When two colors are selected and there is a deficient vertex, it selects a deficient vertex, then it selects one of its edges causing deficiency, then it flips its color. It has probability at least $\frac{1}{4}$ ($\frac{1}{2}$ for selecting uniformly from two deficient vertices and $\frac{1}{2}$ for selecting uniformly from two edges).
    \item When two colors are selected and there is no deficient vertex, then with $\frac{1}{2}$ probability, it selects uniformly a an edge with one of the two prescribed colors, and it flips its color. Observe that the number of edges with any of a set of two colors is at most the number of vertices in an edge coloring. Therefore, the probability of this transformation is at least $\frac{1}{2|V|}$.
    \item When two colors are selected and there is no deficient vertices, it does nothing with probability $\frac{1}{2}$.
    \item When three colors are selected then either a deficient vertex is selected with at least $\frac{1}{2}$ probability or a non-deficient vertex is selected with at least $\frac{1}{2|V|}$ probability. If a deficient vertex is selected, then one of the edges causing deficiency is selected with $\frac{1}{2}$ probability, then colors are flipped along an unequivocally defined alternating path. If a non-deficient vertex is selected, then the colors are flipped along one of at most two paths. The path on which the colors are flipped is selected with at least $\frac{1}{2}$ probability.
\end{enumerate}
Therefore $P(f)$ is indeed lower bounded by $\frac{1}{4|V|}$, thus its inverse is upper bounded by $4|V|$.
\\

In this way, we just proved the following theorem.
\begin{thm}\label{theo:inv-acc-rate-general}
Let $M(G,k)$ be the Markov chain on the edge $k$-colorings on the bipartite graph $G = (V,E)$. Then the inverse of the acceptance ratio in the Metropolis-Hastings algorithm applied with the constant function $g$ and with Markov chain $M(G,k)$ is upper bounded by $96|V|^2(|V|-1)$.
\end{thm}

\section{Latin rectangles}\label{section4}

\begin{dfn}
A \emph{Latin rectangle} is a $k\times n$ table such that each row is a permutation of $[n]$ and each column has no repeats. When $k=n$ we call it a \emph{Latin square}. The \emph{completion} of an $k\times n$ Latin rectangle $R$ is an $n\times n$ Latin square such that its first $k$ rows is $R$.
\end{dfn}

\begin{obs}\label{bij}
For any $(n-k)\times n$ Latin rectangle $R$, there is an $k$-regular bipartite graph $G$ with $n$ vertices in both vertex classes such that completions of $R$ and edge-$k$-colorings of $G$ are in one-to-one correspondence.
\end{obs}
\begin{proof}
Let the two parts of vertices of $G$ be $A_1,A_2,\ldots,A_n$ and $B_1,\ldots,B_n$ and define $A_iB_j\in E(G)$ if and only if number $i$ does not appear in the $j$-th column of $R$. This gives us a desired $k$-regular graph for any given Latin rectangle $R$.

Indeed, it is clear that each $B_i$ is of degree $k$. For $A_i$'s, observe that they all have degree at least $k$. Otherwise, suppose that number $i_0$ is missed in less than $k$ columns (i.e. it appears in at least $n-k+1$ columns). We know that there are only $n-k$ rows, so by the Pigeonhole principle there must be two $i_0$'s in the same row, which contradicts our assumption that $R$ is a Latin rectangle. Because $|E|=nk$ and the total degree of $A_i$'s is $nk$, no $A_i$ can have degree more than $k$ thus every vertex of $G$ is $k$-regular.

Now we give the bijection: if and only if number $i$ is filled in the $(n-k+\ell,j)$-grid in $R$'s completion, we color edge $A_iB_j$ using color $c_{\ell}$ in $G$. For one direction, above discussion for $G$ already shows the pre-image is a completion of $R$. On the other hand, first of all, every $B_j$ is adjacent to one $c_{\ell}$-colored edge for all $k$. Then, by the $k$-regularity of $G$, each number appears $k$ times in the last $k$ rows of the table so each $A_i$ also has exactly one edge for every color. Thus the image is an edge $k$ coloring.
\end{proof}
\begin{rmk}
These objects are also in one-to-one correspondence with half-regular factorizations of complete $(k+n)$-bipartite graphs with edge $n$ colorings, connecting our work to \cite{Aksenetal2017}.
\end{rmk}




%

We now give a simplified version of the Markov chain described in Definition \ref{def:markov-chain} on the solution space of edge $k$-coloring of $k$-regular bipartite graphs, that is, Latin square completions of $(n-k)\times n$ Latin rectangles. In that, we utilize the following observation.

\begin{obs}
Let $C$ be an almost edge $k$-coloring of an $k$ regular bipartite graph. Then there are exactly two deficient vertices in $C$, furthermore, they are the endpoints of a maximal alternating path with two colors.
\end{obs}
\begin{proof}
By definition, there is at least one deficient vertex, $v_1$, in an almost edge $r$-coloring. W.l.o.g., it has $(+c-c')$-deficiency. Let $e$ be one of the edges with color $c$ incident to $v_1$. Consider a longest alternating path with colors $c$ and $c'$ containing $e$. It ends in a vertex $v_2$ either with a $c$ edge, and then there is no $c'$ edge incident to $v_2$, or it ends with a $c'$ edge and then there is no $c$ edge incident to $v_2$. However, there are $k$ colors and $v_2$ has degree $k$, therefore one of the colors must be repeated on edges incident to $v_2$. That is, $v_2$ is a deficient vertex. Also by definition, there cannot be more than two deficient vertices in $C$.
\end{proof}
A consequence is that the $\varphi$ operations always eliminates two deficient vertices by flipping the colors along an alternating path. This allows us to simplify the Markov chain in Definition~\ref{def:markov-chain} to get an irreducible Markov chain on the edge $k$-colorings of $k$-regular bipartite graphs. We also need the following observation.
\begin{obs}\label{obs:no-contradiction-regular}
Let $C$ be an almost edge $k$-coloring of a $k$-regular bipartite graph such that $v_1$ has a $(+c-c')$-deficiency and $v_2$ has a $(+c'-c)$ deficiency. Furthermore, assume that $v_1$ and $v_2$ are in the same vertex class. Let $c'' \ne c, c'$ be a third color. Then the longest alternating walk with colors $c'$ and $c''$ that starts with color $c''$ in $v_2$ is a cycle.
\end{obs}
\begin{proof}
Since the graph is $k$-regular, each non-deficient vertex has exactly one $c'$ edge and one $c''$ edge. That is, the walk continues until it arrives to a deficient vertex or it travels back to $v_2$. Since $v_1$ and $v_2$ are in the same vertex class, the alternating walk could arrive to $v_1$ with a $c'$ edge, however, $v_1$ does not have an incident $c'$ edge. Therefore, the walk arrives back to $v_2$.
\end{proof}
\begin{dfn}
Let $G$ be a $k$-regular equi-bipartite graph. We define a Markov chain $M(G,k)$ on the edge $k$-colorings of $G$. Let the current coloring be $C$. Then draw a random edge $k$-coloring in the following way.
\begin{enumerate}
    \item With probability $\frac{1}{2}$, do nothing (so the defined Markov chain is a Lazy Markov chain).
    \item With probability $\frac{1}{4}$, select two colors $c$ and $c'$ from the set of colors uniformly among the ${k\choose 2}$ possible unordered pairs. Consider the subgraph of $G$ with colors $c$ and $c'$, let it be denoted by $H$. Note that $H$ consists of disjoint alternating cycles. Then select one of them:
    \begin{enumerate}
        \item With probability $\frac{1}{2}$, do nothing.
        \item With probability $\frac{1}{2}$, uniformly select one of the possible connected subgraphs of $H$.
    \end{enumerate}

    Flip the colors in the selected sub-path. If there is any deficient vertex, select one of them uniformly, select the edge with the same color incident to the selected vertex uniformly, and change its color from $c$ to $c'$ or $c'$ to $c$.
    Otherwise, with $\frac{1}{2}$ probability, select either a $c$-edge or a $c'$-edge, uniformly from all edges with color $c$ and $c'$, and change its color to the opposite ($c'$ or $c$), and with $\frac{1}{2}$ probability, do nothing. Eliminate all deficiency by flipping the colors on the appropriate alternating path, selecting uniformly from the two neighbor edges with the same color incident to deficient vertices.
    \item With probability $\frac{1}{4}$ select three colors uniformly from the all possible ${k\choose 3}$ unordered pairs, and select uniformly one from the three colors. Let the distinguished color be denoted by $c''$, and let the other two colors be denoted by $c$ and $c'$. Consider the subgraph consisting of the colors $c$ and $c'$ and let it be denoted by $H$. Its non-trivial components are again copies of cycles. Then select one of the following:
    \begin{enumerate}
        \item With probability $\frac{1}{2}$, do nothing.
        \item With probability $\frac{1}{2}$, uniformly select a sub-path of $H$ with even length (even number of edges).
    \end{enumerate}
    Flip the colors in the selected sub-path. The number of deficiency must be even. If there are two deficient vertices, select one of them uniformly. If there is no deficient vertex, then select uniformly a non-deficient vertex.
    
    If a deficient vertex  $v$ is selected,  let its repeated color be denoted by $\tilde{c}$. Consider the maximal alternating cycle with colors $\tilde{c}$ and $c''$ that contains $v$ (from Observation~\ref{obs:no-contradiction-regular} we know that this cycle exists). Flip the colors in this cycle.
    
    If a non-deficient vertex is selected, let its incident edge with color $c''$ be denoted by $e$. Select uniformly from the incident edges with color $c$ or $c'$, and let it be denoted by $f$ and its color be $\tilde{c}$. Take the maximal alternating cycle with colors $c''$ and $\tilde{c}$ that contains $e$ and $f$. Flip the colors in this cycle.
    
    Eliminate all deficiency by flipping the colors on the appropriate alternating path, selecting uniformly from the two neighbor edges with the same color incident to deficient vertices.
\end{enumerate}
\end{dfn}
We can state the following upper bounds on the diameter and the inverse of acceptance ratio.
\begin{thm}
Let $M(G,k)$ be the Markov chain on the edge $k$-colorings of a $k$-regular bipartite graph $G = (V,E)$. Then the diameter of $M(G,k)$ is upper bounded by $3|E|$ and the inverse of the acceptance ratio in the Metropolis-Hastings algorithm with the uniform distribution is upper bounded by $16|V|(|V|-1)$.
\end{thm}
\begin{proof}
In each subgraph $H_i$, in each non-trivial component half of the edges have color $c_i$ in $C_2$. Indeed, observe that each non-trivial component is an alternating cycle with edges having color $c_i$ in $C_1$ and $C_2$. Therefore, the total number of edges of the non-trivial components in all $H_i$'s is at most $2|E|$. Since for each component the number of necessary transformations is at most $\frac{3}{2}s$, where $s$ is the number of edges in the given component, $3|E|$ number of steps are sufficient to transform any edge coloring to any other one.

We can decompose the proposal and backproposal probabilities in the same way as we did in case of Theorem~\ref{theo:inv-acc-rate-general}. Then the probabilities of selecting colors again cancel. We still have the bound 
$$
\frac{P(\varphi_2^{-1})}{P(\varphi_2)} \le \frac{1}{4}.
$$
However, we can give a better upper bound on
$$
\frac{P(\varphi_1)P(f^{-1})}{P(\varphi_1^{-1}P(f))} \le \frac{1}{P(\varphi_1^{-1}P(f))} \le 8 {|V|\choose 2}.
$$
Indeed, since $G$ is $k$-regular and its edges are colored with $k$ colors, there are deficient vertices after transformation $\varphi_1^{-1}$ if and only if an alternating path is selected and the edges are flipped. The probability for that is at least $\frac{1}{2{|V|\choose 2}}/\frac{1}{2}$ for choosing the option to select a path, and there are at most ${|V|\choose 2}$ possible paths, since the two end vertices define the alternating path unequivocally (recall that the two colors in the alternating path are fixed). In such a case, there are two deficient vertices, one of them is selected uniformly and one of its edges with duplicated colors is selected uniformly, and the color of this edge is changed in a prescribed way. Note that this change is the $f$ transformation. Therefore, in this case, $P(f) = \frac{1}{4}$. We get that 
\begin{equation}
P(\varphi_1^{-1}P(f)) \ge \frac{1}{8{|V|\choose 2}} \label{eq:combined-case-I}
\end{equation}

If no path is selected, then the probability for is $\frac{1}{2}$. Then a non-deficient vertex is selected uniformly, its probability is $\frac{1}{|V|}$. Then one of the edges incident to the selected vertex and having one of the two prescribed colors is selected uniformly. This has probability $\frac{1}{2}$. The color of the selected edge is changed in a prescribed way. Therefore, we get that  
\begin{equation}
P(\varphi_1^{-1}P(f)) \ge \frac{1}{4|V|} \label{eq:combined-case-II}
\end{equation}
Since $|V| \ge 2$, we have $4|V| \le 8 {|V|\choose 2}$.

Putting the inequalities together, we obtain the claimed bound.
\end{proof}

\section{Conclusion}
We considered the solution space of edge-colorings of any general bipartite graph and explicitly constructed an irreducible Markov chain $M(G,k)$ on the edge $k$-colorings of a bipartite graph $G$. We showed that the diameter of $M$ is linearly bounded by the number of edges, and when we apply the Metropolis-Hastings algorithm to it so that the modified chain, $\tilde{M}(G,k)$, converges to the uniform distribution, the inverse of acceptance ratio is bounded by a cubic function of the number of vertices. A special case of our work provides a Markov chain Monte Carlo method to sample completions of Latin squares.

Possible further work includes investigating the speed of convergence of the Markov chain $\tilde{M}(G,k)$, that is, whether it is rapidly mixing or not. 
The natural conjecture is that $\tilde{M}(G,k)$ is rapidly mixing, based on the proved properties on the diameter and the bound on the inverse of the acceptance ratio. The Markov chain $\tilde{M}(G,k)$ is similar to the Markov chains invented by Jacobson and Matthew \cite{JacobsonMatthew1996} and Aksen \emph{et al.} \cite{Aksenetal2017}. Neither of these chains are proved to be rapidly mixing, although rapid mixing is conjectured. The fact that the rapid mixing is a $25$ year old open question in case of the Jacobson-Matthew Markov chain on Latin squares indicates that resolving these open questions might be extremely hard. A special case of the Markov chain introduced in this paper is the Markov chain on completions $(n-3)\times n$ Latin rectangles, or equivalently, on edge $3$-colorings of $3$-regular bipartite graphs. The structure of the solution space of edge $3$-colorings of $3$-regular bipartite graphs might be simpler than the structure of the Latin squares. Therefore, there is a hope that proving rapid mixing of the Markov chain on edge $3$-colorings of $3$-regular bipartite graphs might be easier.

\section*{Acknowledgement}
I.M. was supported by NKFIH grants KH126853, K132696 and SNN135643. Dr. C.F. Bubb is thanked for his support. The project is a continuation of the work done at the 2019 Budapest Semesters in Mathematics. Both authors would like to thank the BSM for running the program.

\bibliographystyle{elsarticle-num}

\end{document}